\documentclass[11pt]{amsart}

\usepackage{bbm}
\usepackage{amssymb,tikz,amsthm,amstext,amscd,epsfig,euscript, mathrsfs,bbm,dsfont,pspicture,multicol,graphpap,graphics,graphicx,times,enumerate,subfig,sidecap,wrapfig,color}

\usepackage{pdfsync}

\newcommand{\R}{{\mathbb R}}       

\newcommand{\FF}{{\mathcal F}}
\newcommand{\HH}{{\mathcal H}}

\newcommand{\EE}{{\mathcal E}}

\newcommand{\diam}{\operatorname{diam}}

\newcommand{\rf}[1]{{(\ref{#1})}}

\newcommand{\supp}{\operatorname{supp}}

\newcommand{\vphi}{{\varphi}}
\newcommand{\ve}{{\varepsilon}}
\newcommand{\vv}{{\vspace{2mm}}}
\newcommand{\vvv}{{\vspace{3mm}}}
\newcommand{\wt}[1]{{\widetilde{#1}}}

\newcommand{\bad}{{\rm Bad}}
\newcommand{\uss}{{\rm Uss}}

\newtheorem{theorem}{Theorem}[section]
\newtheorem{lemma}[theorem]{Lemma}

\newtheorem*{theorem*}{Theorem}
\newtheorem*{claim*}{Claim}

\theoremstyle{definition}

\theoremstyle{remark}
\newtheorem{rem}[theorem]{Remark}

\numberwithin{equation}{section}

\newcommand{\brem}{\begin{rem}}
\newcommand{\erem}{\end{rem}}

\newcommand{\capp}{\operatorname{Cap}}

\begin{document}
\title[Harmonic measure and Hausdorff measure]{The mutual singularity of harmonic measure and Hausdorff measure of codimension smaller than one}

\author{Xavier Tolsa}

\address{Xavier Tolsa
\\
ICREA, Passeig Llu\'{\i}s Companys 23 08010 Barcelona, Catalonia\\
 Departament de Matem\`atiques, and BGSMath
\\
Universitat Aut\`onoma de Barcelona
\\
08193 Bellaterra (Barcelona), Catalonia.
}
\email{xtolsa@mat.uab.cat} 

\thanks{Partially supported by by 2017-SGR-0395 (Catalonia) and MTM-2016-77635-P,  MDM-2014-044 (MINECO, Spain).
}

\begin{abstract}
Let $\Omega\subset\R^{n+1}$ be open and let $E\subset \partial\Omega$ with $0<\HH^s(E)<\infty$, for some $s\in(n,n+1)$,
satisfy a local capacity density condition.
In this paper it is shown that the harmonic measure cannot be mutually absolutely continuous
with the Hausdorff measure $\HH^s$ on $E$. This answers a question of Azzam and Mourgoglou, who had proved the same result under
the additional assumption that $\Omega$ is a uniform domain.
\end{abstract}

\maketitle

\section{Introduction}

In this paper we study the relationship between harmonic measure and Hausdorff measure of codimension smaller than $1$ in $\R^{n+1}$. The importance of harmonic measure is mainly due to its connection with the Dirichlet problem for the Laplacian. Indeed, recall that given a domain $\Omega\subset\R^{n+1}$ and a point $p\in\Omega$, the harmonic measure with pole at $p$ is the measure $\omega^p$ satisfying the property that, 
for any function $f
\in C(\partial\Omega)\cap L^\infty(\partial\Omega)$, the integral
$\int f\,d\omega^p$
equals the value at $p$ of the harmonic extension of $f$ to $\Omega$.

The study of the metric and geometric properties of harmonic measure has been a classical subject in 
mathematical analysis since the Riesz brothers theorem \cite{RR} asserting that harmonic measure is absolutely continuous with respect to arc length measure on simply connected planar domains with rectifiable boundary.
In the plane, complex analysis plays a very important role in connection with harmonic measure, essentially because of the invariance
of harmonic measure by conformal mappings. This fact makes the case of planar domains rather special.

In the plane it is known that the dimension of harmonic measure is at most $1$ by a celebrated result of
Jones and Wolff \cite{JW}. This means that there exists a set $E\subset \partial\Omega$ of Hausdorff dimension at most $1$ which has full harmonic measure. Furthermore, such set $E$ can be taken so that it has $\sigma$-finite length, as shown by Wolff \cite{Wolff-plane}. More precise results for simply connected planar domains had been obtained previously by Makarov \cite{Mak1}, \cite{Mak2}.

In higher dimensions one has to use real analysis techniques to study harmonic measure. The codimension $1$ is still quite special, mainly because of the relationship between harmonic measure and rectifiability. For instance, in \cite{AHM3TV}
it was shown that the mutual absolute continuity between harmonic measure and $n$-dimensional Hausdorff measure on a subset 
$E\subset\partial\Omega$, $\Omega\subset\R^{n+1}$, implies the $n$-rectifiability of $E$. 
Also, under the assumption that
$\partial \Omega$ is $n$-AD-regular, that is $\HH^n(B(x,r)\cap\partial\Omega)\approx r^n$ for all $x\in\partial\Omega$, $0<r\leq\diam(\partial\Omega)$,
many recent works have been devoted to relate quantitative properties of harmonic measure and other analytic or geometric properties of the domain.
See for example  \cite{Azzam-semi},  \cite{GMT}, \cite{HLMN}, \cite{HM1}, \cite{HM2}, \cite{HMM}, \cite{HMU}, \cite{MT}.

One of the main differences between the planar case and the higher dimensional case is that in $\R^{n+1}$, with $n\geq2$, there exist domains where the dimension of harmonic measure is larger than $n$. This was proved by Wolff in \cite{Wolff-counterexamples}. An important open problem consists of finding the sharp value for the upper bound of 
the dimension of harmonic measure in $\R^{n+1}$, $n\geq2$. In \cite{Bourgain}, Bourgain showed that this sharp value is strictly smaller than $n+1$. In \cite{Jones-scaling}, Jones conjectured that the sharp bound should be $n+1-1/n$. 
However, for the moment there have been no significative advances on this open problem.
On the other hand, the techniques of Bourgain \cite{Bourgain} play an important role in more
recent results asserting that in some classes of sets (for example, in some self-similar sets) the dimension
of harmonic measure is strictly smaller than the dimension of the set. See \cite{Batakis1}, \cite{Batakis2}, and \cite{Azzam-drop}.

As mentioned above, the current paper deals with harmonic measure in the case of codimension less than $1$. Although the main result of the paper is not directly related to the above Jones' conjecture, I think that this contributes to a better understanding of
the behavior of harmonic measure in this codimension.

To state precisely the main result, we need some additional notation.
For $n\geq2$, let $\Omega\subset\R^{n+1}$ be open and let $E\subset \partial\Omega$ be a non-empty set.
We say that the local capacity density condition (or local CDC) holds in $E$ is there exists constants $r_E>0$ and $c_E>0$ such that
\begin{equation}\label{eqloccdc}
\capp(B(x,r)\cap \Omega^c)\geq c_E\,r^{n-1}\quad\mbox{ for all $x\in E$ and $0<r\leq r_E$,}
\end{equation}
where $\capp$ stands for the Newtonian capacity (see Section \ref{secnew} for the definition).
We denote by $\omega$ the harmonic measure in $\Omega$.

The main result of this paper is the following.

\begin{theorem}\label{teo1}
Given $n\geq2$ and $s>n$, let $\Omega\subset \R^{n+1}$ be open and let $E\subset\partial\Omega$ be such 
$\HH^s(E)<\infty$. Suppose that the harmonic measure $\omega$ and the Hausdorff measure $\HH^s$ are mutually absolutely continuous
in $E$ and that the local CDC holds in $E$. Then $\HH^s(E)=\omega(E)=0$.
\end{theorem}


In other words, harmonic measure cannot be mutually absolutely continuous with Hausdorff measure of codimension less than $1$ in any subset of positive harmonic measure, under the local CDC assumption.
Recall that the same result was proved in \cite{AM} by Azzam and Mourgoglou under the additional
assumption that $\Omega$ is a uniform domain. Recall also that, roughly speaking, a domain is called uniform if it satisfies
an interior porosity assumption (the so-called interior corkscrew condition), and a quantitative connectivity condition
in terms of Harnack chains.

The methods in the current paper are very different from the ones used in \cite{AM}. The new main tool is an identity obtained by integration by parts (see Section \ref{seckey}), whose application requires later some rather delicate
stopping time arguments. On the other hand, the arguments in \cite{AM} use blowups and tangent measures,
and it seems that the uniformity assumption is important. Indeed, in their work, Azzam and Mourgoglou
leave as an open question the possibility of eliminating the uniformity assumption. They also ask the same questions about the CDC: can this be avoided?
While Theorem \ref{teo1} confirms that uniformity is not necessary, it is still an open problem to know if the CDC is required.

In fact, in \cite{AM} 
a non-degeneracy condition weaker (at least, a priori) than the CDC is used. 
I think that, quite likely, in the arguments below one may be able to replace the local CDC assumption by the non-degeneracy
condition of Azzam-Mourgoglou. However, I have preferred to state Theorem \ref{teo1} in terms of the local
CDC, which is closer to the usual CDC. 
On the other hand, the techniques in the current paper do not look very useful for codimensions larger than $1$, unlike the arguments in \cite{AM}, which are applied by the authors to derive other related results.

Nevertheless, it seems the statement in Theorem \ref{teo1} does not hold for $s<n$. Actually,
according to an example constructed by Alexander Volberg, there is a domain $\Omega\subset\R^2$ satisfying the CDC such that harmonic measure is mutually absolutely continuous with $\HH^s|_{\partial\Omega}$ for some $0<s<1$ (cf. \cite{AM}). 

The aforementioned integration by parts formula (see \rf{eqa8}) required for the proof of Theorem \ref{teo1} is a generalization of a formula that has already been applied to some problems involving harmonic or elliptic measure and rectifiability in works
such as \cite{HLMN} or \cite{AGMT}, and it goes back to some work of Lewis and Vogel \cite{LV}, at least.



\section{Preliminaries}

In the paper, constants denoted by $C$ or $c$ depend just on the dimension and perhaps other fixed
parameters (such as the constant $c_E$ involved the local CDC, for example). We will write $a\lesssim b$ if there is $C>0$ such that $a\leq Cb$ . We write $a\approx b$ if $a\lesssim b\lesssim a$.

\subsection{Measures}

The set of (positive) Radon measure in $\R^{n+1}$ is denoted by $M_+(\R^{n+1})$. The Hausdorff $s$-dimensional measure and Hausdorff $s$-dimensional content are denoted ty $\HH^s$ and $\HH^s_\infty$, respectively.

Given $\mu\in M_+(\R^{n+1})$, its supper $s$-dimensional density at $x$ is defined by
$$\Theta^{s,*}(x,\mu) = \limsup_{r\to 0} \frac{\mu(B(x,r))}{(2r)^s}.$$
Recall that, given an $\HH^s$-measurable set $E\subset\R^{n+1}$ with $0<\HH^s(E)<\infty$, we have
\begin{equation}\label{eqdens*}
2^{-s}\leq \Theta^{s,*}(x,\HH^s|_E)\leq1\quad\mbox{ for $\HH^s$-a.e.\ $x\in E$.}
\end{equation}
See \cite[Chapter 6]{Mattila-book}, for example.


\subsection{Newtonian capacity and harmonic measure}\label{secnew}

The fundamental solution of the negative Laplacian in $\R^{n+1}$, $n\geq2$, equals
$$\EE(x) = \frac{c_n}{|x|^{n-1}},$$
where $c_n= (n-1)\HH^n(\mathbb S^n)$, with $\mathbb S^n$ being the unit hypersphere in $\R^{n+1}$.

The Newtonian potential of a measure $\mu\in M_+(\R^{n+1})$ is defined by
$$U\mu(x) = \EE * \mu(x),$$
and the Newtonian capacity of a compact set $F\subset\R^{n+1}$ equals
$$\capp(F)=\sup\big\{\mu(F):\mu\in M_+(\R^{n+1}),\,\supp\mu\subset F,\|U\mu\|_{\infty}\leq1\big\}.$$
It is well known that 
$$\|U\mu\|_{\infty} = \|U\mu\|_{\infty,F},$$
and that there exist a unique measure that attains the supremum in the definition of $\capp(F)$. 
If $\mu$ attains that supremum, then it turns out that
$U\mu(x)=1$ for quasievery $x\in F$ (denoted also q.e. in $F$),  i.e., for all $x\in F$
with the possible exception of a set of zero Newtonian capacity.
The probability measure 
$$\mu_F = \frac1{\capp(F)}\,\mu$$ is called equilibrium measure (of $F$), and it holds that
$$U\mu_F(x) = \frac1{\capp(F)}\quad \mbox{ for q.e.\ $x\in F$.}$$

Recall that we denote by $\omega$ the harmonic measure on an open set $\Omega$. The associated
Green function is denoted by $g(\cdot,\cdot)$. The following result is quite well known, but we
prove it here for the reader's convenience.

\begin{lemma}
\label{lembourgain}
 Given $n\geq2$, let $\Omega\subset \R^{n+1}$ be open and let $B$ be a closed ball centered at $\partial\Omega$. Then 
\[ \omega^{x}(B)\geq c(n) \frac{\capp(\tfrac14 B\cap\partial\Omega)}{r(B)^{n-1}}\quad \mbox{  for all }x\in \tfrac14 B\cap \Omega ,\]
with $c(n)>0$.
\end{lemma}

\begin{proof}
Let $\mu_{\frac14 B\cap\partial\Omega}$ be the equilibrium measure for $\frac14 B\cap\partial\Omega$, and
let $\mu = \capp(\frac14 B\cap\partial\Omega)\,\mu_{\frac14 B\cap\partial\Omega}$, 
so that $\|U\mu\|_\infty\leq1$ and $\|\mu\|=\capp(\frac14 B\cap\partial\Omega)$. 
Notice that, for all $x\in B^c$,
$$U\mu(x) =\int\frac{c_n}{|x-y|^{n-1}}\,d\mu(y) \leq \frac{c_n\|\mu\|}{(\frac34 r(B))^{n-1}}.
$$

Consider the function 
$f(x) = U\mu(x) - \frac{c_n\|\mu\|}{(\frac34 (B))^{n-1}}$. Using that $f(x)\leq 0$ in $B^c$, $\|f\|_\infty\leq1$, and that $f$ is harmonic in $\Omega$, by the maximum principle we deduce that, for all $x\in \Omega$,
$$\omega^x(B)\geq f(x).$$
In particular, for $x\in \frac14B$ we have
\begin{align*}
\omega^x(B)&\geq \int\frac{c_n}{|x-y|^{n-1}}\,d\mu(y) - \frac{c_n\|\mu\|}{(\frac34 r(B))^{n-1}}\\
& \geq  \frac{c_n\|\mu\|}{(\frac12 r(B))^{n-1}}  - \frac{c_n\|\mu\|}{(\frac34 r(B))^{n-1}} = c_n \big(2^{n-1} - (\tfrac43)^{n-1}\big) \,\frac{\capp(\frac14 B\cap\partial\Omega)}{r(B)^{n-1}} ,
\end{align*}
which proves the lemma.
\end{proof}

\vv
We recall also the following lemma, whose proof can be found in \cite{AHM3TV}.

\begin{lemma}\label{l:w>G}
Let $n\ge 2$ and $\Omega\subset\R^{n+1}$ be open.
Let $B$ be a closed ball centered at $\partial\Omega$. Then, for all $a>0$,
\begin{equation}\label{eq:Green-lowerbound1}
 \omega^{x}(aB)\gtrsim \inf_{z\in 2B\cap \Omega} \omega^{z}(aB)\, r(B)^{n-1}\, g(x,y)\quad\mbox{
 for all $x\in \Omega\backslash  2B$ and $y\in B\cap\Omega$,}
 \end{equation}
 with the implicit constant independent of $a$.
\end{lemma}
 
Combining the two preceding lemmas, choosing $a=8$, we obtain:

\begin{lemma}\label{l:w>GG}
Let $n\ge 2$, $s>n-1$, and $\Omega\subset\R^{n+1}$ be open.
Let $B$ be a closed ball centered at $\partial\Omega$. Then, 
\begin{equation}\label{eqgreen2}
 \omega^{x}(8B)\gtrsim_{n} \capp(2 B\cap\partial\Omega)\, r(B)^{n-1}\, g(x,y)\quad\mbox{
 for all $x\in \Omega\backslash  2B$ and $y\in B\cap\Omega$.}
 \end{equation}
\end{lemma}

\vv


\section{The key identity and the main idea}

\subsection{The key identity}\label{seckey}

\begin{lemma}[Key identity]
Let $\Omega\subset\R^{n+1}$ be open, let $\psi\in C_c^\infty(\Omega)$, and let $u:\Omega\to\R$ be harmonic and positive in $\supp\psi$.
Then, for any $\alpha\in\R$,
\begin{align}\label{eqa8}
\int |\nabla^2 u|^2 \,u^\alpha\,\psi\,dx & = \frac12\,\alpha (\alpha-1) \int |\nabla u|^4\,
u^{\alpha-2}\,\psi\,dx \\
&\quad - \frac12 \int \nabla(|\nabla u|^2)\cdot\nabla\psi\, u^\alpha\,dx +
\frac12 \int |\nabla u|^2\,\nabla(u^\alpha)\cdot\nabla\psi\,dx.\nonumber
\end{align}
\end{lemma}

In the lemma we denoted
$$|\nabla^2 u|^2 = \sum_{i,j}(\partial_{i,j} u)^2.$$
The identity \rf{eqa8}, in the particular case $\alpha=1$, was already used in connection with harmonic measure in \cite{LV} and \cite{HLMN}.

\begin{proof}
Notice that
$$|\nabla^2 u|^2 = \sum_{i}|\nabla \partial_{i} u|^2.$$
So \rf{eqa8} follows by summing from $i=1$ to $n+1$ the following identity:
\begin{align}\label{eqa9}
\int |\nabla\partial_i u|^2 \,u^\alpha\,\psi\,dx & = \frac12\,\alpha (\alpha-1) \int |\partial_i u|^2\,|\nabla u|^2\,
u^{\alpha-2}\,\psi\,dx \\
&\quad - \frac12 \int \nabla(|\partial_i u|^2)\cdot\nabla\psi\, u^\alpha\,dx +
\frac12 \int |\partial_i u|^2\,\nabla(u^\alpha)\cdot\nabla\psi\,dx.\nonumber
\end{align}
To prove this, we integrate by parts:
\begin{align*}
\int |\nabla\partial_i u|^2 \,u^\alpha\,\psi\,dx & =
\int \nabla\partial_i u \cdot \nabla\partial_i u \,\,u^\alpha\,\psi\,dx \\
& = \int \nabla\partial_i u \cdot \nabla\big(\partial_i u \,u^\alpha\,\psi\big)\,dx
- \int \nabla\partial_i u \cdot \nabla\big(u^\alpha\,\psi\big)\,\partial_i u\,dx.
\end{align*}
The first integral on the right hand side vanishes because $u$ is harmonic:
$$\int \nabla\partial_i u \cdot \nabla\big(\partial_i u \,u^\alpha\,\psi\big)\,dx
=- \int \Delta(\partial_i u) \,\big(\partial_i u \,u^\alpha\,\psi\big)\,dx=0.$$
Using also
$\partial_i u\,\nabla\partial_i u = \frac12 \nabla(|\partial_i u|^2)$, we get
\begin{align}\label{eq48}
\int |\nabla\partial_i u|^2 \,u^\alpha\,\psi\,dx  & = 
-\frac12 \int \nabla(|\partial_i u|^2)\cdot \nabla\big(u^\alpha\,\psi\big)\,dx\\
& = -\frac12 \int \nabla(|\partial_i u|^2)\cdot \nabla\big(u^\alpha\big)\,\psi\,dx
-\frac12 \int \nabla(|\partial_i u|^2)\cdot \nabla\psi\, u^\alpha\,dx\nonumber\\
& = 
-\frac12 \int \nabla(|\partial_i u|^2\,\psi)\cdot \nabla\big(u^\alpha\big)\,dx
 +\frac12 \int |\partial_i u|^2 \,\nabla\psi\cdot \nabla\big(u^\alpha\big)\,dx\nonumber\\
&\quad-\frac12 \int \nabla(|\partial_i u|^2)\cdot \nabla\psi\, u^\alpha\,dx.\nonumber
\end{align}
Finally, integrating by parts and taking into account that 
$\Delta\big(u^\alpha\big)=\alpha(\alpha-1)|\nabla u |^2 \,u^{\alpha-2}$, we deduce that
 the first term on the right hand side satisfies
$$-\frac12 \int \nabla(|\partial_i u|^2\,\psi)\cdot \nabla\big(u^\alpha\big)\,dx =
\frac12 \int |\partial_i u|^2\,\psi\,\Delta\big(u^\alpha\big)\,dx = 
\frac12\,\alpha(\alpha-1) \int |\partial_i u|^2\,\psi\,|\nabla u |^2 \,u^{\alpha-2}\,dx.$$
Plugging this into \rf{eq48}, we get \rf{eqa9}.
\end{proof}

\subsection{The strategy of the proof}
Let $s>n$ be as in Theorem \ref{teo1}. By Bourgain's theorem \cite{Bourgain}, it is clear that we 
can assume $s\in (n,n+1)$.
Let $a\in(0,1)$ be such that $s=n+a$, and let 
$$\alpha = \frac{1-a}{1+a},$$
so that $\alpha\in (0,1)$ too.
We will apply the identity \rf{eqa8} with $u$ equal to the Green function $g(\cdot,p)$ and a suitable function $\psi$. The choice of the preceding value of $\alpha$ is motivated by the fact that then
the integrals that appear in \rf{eqa8} scale like
$$\frac1{\ell^4}\left(\frac{\omega(\cdot)}{\ell^{n-1}}\right)^{\alpha+2}\ell^{n+1} =
\omega(\cdot) \,\left(\frac{\omega(\cdot)}{\ell^s}\right)^{\alpha+1},$$
under the assumption that that $u=g(\cdot,p)$ scales like $\omega(\cdot)\ell^{1-n}$.

A key fact in our arguments is that the first term on the right hand side of \rf{eqa8} is negative (because
$\alpha(\alpha-1)<0$), while
the left hand side is positive. These two terms should be considered as the main ones in \rf{eqa8}, and the last two integrals should be considered as ``boundary terms" because of the presence of $\nabla\psi$ in their integrands. 

Writing
$g(x)=g(x,p)$, from \rf{eqa8} we get
\begin{align}\label{eqa88a}
 |\alpha (\alpha-1&)| \int |\nabla g|^4\,
g^{\alpha-2}\,\psi\,dx \\
& \leq 
 \left|\int \nabla(|\nabla g|^2)\cdot\nabla\psi\, g^\alpha\,dx\right| +
\left|\int |\nabla g|^2\,\nabla(g^\alpha)\cdot\nabla\psi\,dx\right| - 2\int |\nabla^2 g|^2 \,g^\alpha\,\psi\,dx.\nonumber
\end{align}
Using this inequality and assuming the existence of a set $E\subset\partial\Omega$ with $\omega(E)>0$ such that the harmonic measure and the Hausdorff measure $\HH^s$ are absolutely continuous on $E$, we will get a contradiction.
To this end, we will construct an appropriate function $\psi$ by some stopping time arguments involving the set $E$,
and with this choice we will show that the integral on left hand side of \rf{eqa88a} is much larger 
than the right hand side.
\vv

To illustrate how we will apply the inequality \rf{eqa88a} we consider a simple example. Suppose that $E\subset B(0,1)$ is compact and $s$-AD-regular (with $s\in (n,n+1)$), that is, $\HH^s(E\cap B(x,r))\approx r^s$ for all $x\in E, \;0\leq r\leq \diam E$. Let $\Omega=\R^{n+1}\setminus E$ and suppose that $\diam E\approx 1$. Observe that the $s$-AD-regularity of $E$ ensures that the CDC holds.
We will sketch how one can check that the harmonic measure $\omega$ for $\Omega$ is not comparable to $\HH^s|_E$, that is, $\omega$ cannot be of the form
$\omega= h\,\HH^s|_E$, for some density function $h$ such that $h\approx 1$. For the sake of contradiction, assume $\omega= h\,\HH^s|_E$, with $h\approx 1$ (of course, this condition is much stronger than the mutual absolute continuity of $\omega$ and $\HH^s|_E$, but this suffices for the example).

Given a small parameter $r_0\in (0,1/10)$, 
consider a $C^\infty$ function $\psi$
satisfying $\chi_{B(0,2)\setminus U_{r_0}(E)}\leq\psi\leq \chi_{B(0,3)\setminus U_{r_0/2}(E)}$ (where $U_\rho(E)$ is the $\rho$-neighborhood of $E$), with $|\nabla \psi|\lesssim 1/r_0$ in $U_{r_0}(E)$, and $|\nabla \psi|\lesssim 1$ in $B(0,3)\setminus B(0,2)$. Suppose that the pole $p$ for harmonic measure is far from $E$, say $p\not \in B(0,4)$.
Consider the second integral on the right hand side of \rf{eqa88a}. Notice that
\begin{equation}\label{eq1*2}
\left|\int |\nabla g|^2\,\nabla(g^\alpha)\cdot\nabla\psi\,dx\right|\lesssim \frac1{r_0} \int_{U_{r_0}(E)\setminus  U_{r_0/2}(E)} |\nabla g|^3\,g^{\alpha-1}\,dx +  \int_{B(0,3)\setminus B(0,2)} |\nabla g|^3\,g^{\alpha-1}\,dx.
\end{equation}
To estimate the first integral, consider a finite family of balls $B_h$, $h\in H$, with radii $2r_0$, centered at $E$, which cover $U_{r_0}(E)$ and have bounded overlap.
Then, for each ball $B_h$, using the harmonicity of $g$ and Lemma \ref{l:w>GG},
\begin{align*}
\frac1{r_0} \int_{B_h\setminus U_{r_0/2}(E)} |\nabla g|^3\,g^{\alpha-1}\,dx & \lesssim 
\frac1{r_0^4} \int_{B_h} g^{\alpha+2}\,dx \lesssim r_0^{n-3} \bigg(\frac{\omega(10B_h)}{r_0^{n-1}}\bigg)^{\alpha+2} \\ & = \omega(10B_h) \bigg(\frac{\omega(10B_h)}{r_0^s}\bigg)^{\alpha+1}\approx \HH^s(B_h\cap E).
\end{align*}
Thus,
\begin{align*}
\frac1{r_0} \int_{U_{r_0}(E)\setminus  U_{r_0/2}(E)} |\nabla g|^3\,g^{\alpha-1}\,dx 
& \leq \sum_{h\in H} \frac1{r_0} \int_{B_h\setminus U_{r_0/2}(E)} |\nabla g|^3\,g^{\alpha-1}\,dx\\
&\lesssim \sum_{h\in H} \HH^s(B_h\cap E) \lesssim \HH^s(E).
\end{align*}
By analogous arguments, one can show that the last integral on the right hand side is at most $C\HH^s(E)$,
and also the first integral on the right hand side of \rf{eqa88a}. These estimates together with \rf{eqa88a} imply that
\begin{equation}\label{eqdfj21}
\int |\nabla g|^4\,
g^{\alpha-2}\,\psi\,dx \lesssim \HH^s(E).
\end{equation}

To reach the desired contradiction it is reasonable to try to estimate the integral on the left hand side above from below. To this end, consider a ball $B$ centered at $E$ with radius
$r(B)\in (\Lambda r_0,1)$, for some constant $\Lambda >1$. One can show that, for $\Lambda$ big enough, the following holds:
\begin{align}\label{eqal932}
\int_{B\setminus U_{\Lambda ^{-1}r(B)}(E)} |\nabla g|^4\,
g^{\alpha-2}\,\psi\,dx & = \int_{B\setminus U_{\Lambda^{-1}r(B)}(E)} |\nabla g|^4\,
g^{\alpha-2}\,dx \\
&\gtrsim \omega(\tfrac12B) \bigg(\frac{\omega(\tfrac12B)}{r(B)^s}\bigg)^{\alpha+1} \approx \HH^s(B\cap E).\nonumber
\end{align}
The detailed proof of this estimate is not too difficult but it would lead us too far. So we will just mention that this follows using Lemma \ref{l:w>GG} and other rather standard arguments
(see Lemma \ref{lemesti0} below for more details).
Given $r\in (\Lambda r_0,1/2)$,
by covering $U_r(E)$ by a family of balls centered at $E$ of radius $2r$ with bounded overlap and applying \rf{eqal932} to each ball, we infer that
$$\int_{U_r(E)\setminus U_{(2\Lambda)^{-1}r}(E)} |\nabla g|^4\,
g^{\alpha-2}\,\psi\,dx \gtrsim \HH^s(E).$$
So assuming that $r_0$ is of the form $r_0=(2\Lambda)^{-N}$, $N>1$, we get
$$\int |\nabla g|^4\,
g^{\alpha-2}\,\psi\,dx \geq \sum_{k=1}^{N-2} 
\int_{U_{(2\Lambda)^{-k}}(E)\setminus U_{(2\Lambda)^{-k-1}}(E)} |\nabla g|^4\,
g^{\alpha-2}\,\psi\,dx
 \gtrsim (N-2)\,\HH^s(E),$$
which contradicts \rf{eqdfj21} if $N$ is big enough, or equivalently, $r_0$ is small enough.

\vv
The proof of Theorem \ref{teo1} will involve similar ideas to the ones above, but with additional technical complications which will require the use of stopping time arguments.

\vv


\section{The ball $B_0$, the stopping construction, and the function $\psi$}\label{secpsi}

\subsection{The ball $B_0$}

From now we assume that we are under the assumptions of Theorem \ref{teo1}. We consider a point $p\in\Omega$ and we denote
by $\omega$ the harmonic measure for $\Omega$ with respect to the pole $p$.
We also denote $\mu = \HH^s|_{E}$ and we assume that $0<\mu(E)<\infty$ and that $\mu$ is absolutely continuous
with respect to $\omega$. Our objective is to find a contradiction. 

 By replacing $E$ by a suitable subset if necessary, by standard methods
(taking into account the upper bound for the upper density of $\mu$ in \rf{eqdens*})
we may assume that there exists some $\delta_0>0$ such that
$$\mu(B(x,r))\leq 3^s\,r^s\quad\mbox{ for all $x\in E$ and $0<r\leq \delta_0$}.$$

Since $\mu\ll\omega$, there exists some non-negative function $h\in L^1(\omega)$ such that $\mu= h\,\omega$.
We consider a point $x_0\in E$ satisfying the following: 
$x_0$ is a Lebesgue point for $h$ with $h(x_0)>0$ and a density point of $E$ (both with respect to $\omega$), and
there exists a sequence of radii $r_k\to 0$ such that 
\begin{equation}\label{eqrkk}
\omega(B(x_0,200r_k))\leq (200)^{n+2}\,\omega(B(x_0,r_k)).
\end{equation}
For this last property, see, for example, Lemma 2.8 in \cite{Tolsa-llibre}.
Now, given some $\kappa_0\in (0,1/10)$, let $\delta_1\in(0,\delta_0]$ be such that 
\begin{equation}\label{eqrkk1}
\frac1{\omega(B(x,r))}\int_{B(x,r)} |h-h(x_0)|\,d\omega \leq \kappa_0\,h(x_0)\,\quad\mbox{ for all
$r\in(0,\delta_1]$}
\end{equation} 
and 
\begin{equation}\label{eqrkk2}
\omega(B(x,r)\setminus E)\leq \kappa_0\,\omega(B(x,r))\,\quad\mbox{ for all
$r\in(0,\delta_1]$.}
\end{equation}
The parameter $\kappa_0$ will be fixed below, and depends only on $n$.
We take now a radius $$\wt r\in \big(0,\,\tfrac1{300}\min(r_E,\delta_1,|p-x_0|)\big)$$ such that \rf{eqrkk} holds for $\wt r=r_k$
(recall that $r_E$ is defined by local CDC in \rf{eqloccdc}), and we denote
$$B_0 = B(x_0,2\wt r).$$
From \rf{eqrkk1} we deduce that, for all $r\in(0,100r(B_0)]$,
$$
\mu(B(x_0,r)) = \int_{B(x_0,r)} h\,d\omega \leq h(x_0)\,\omega(B(x_0,r)) + \int_{B(x_0,r)}|h-h(x_0)|\,d\omega \leq 2\,h(x_0)\,\omega(B(x_0,r)).
$$
Analogously,
$$
\mu(B(x_0,r))  \geq h(x_0)\,\omega(B(x_0,r)) - \int_{B(x_0,r)}|h-h(x_0)|\,d\omega \geq \frac12\,h(x_0)\,\omega(B(x_0,r)).
$$

We collect some of the properties about $B_0$ in the next lemma.

\begin{lemma}\label{lemb0}
For all $r\in (0,100r(B_0)]$, we have
\begin{equation}\label{eqsau2}
\frac12\,h(x_0)\,\omega(B(x_0,r))\leq \mu(B(x_0,r))\leq 2\,h(x_0)\,\omega(B(x_0,r)).
\end{equation}
We also have
\begin{equation}\label{eqsau3}
\omega(100B_0)\leq (200)^{n+1}\,\omega(\tfrac12B_0)
\quad \text{ and }\quad \mu(100B_0)\leq 4(200)^{n+1}\,\mu(\tfrac12B_0).
\end{equation}
\end{lemma}

\begin{proof}
The estimates in \rf{eqsau2} have been shown above, as well as the first inequality in \rf{eqsau3}.
The second inequality follows from the preceding estimates:
$$\mu(100B_0)\leq 2\,h(x_0)\,\omega(100B_0) \leq 2\,h(x_0)\, (200)^{n+1}\,\omega(\tfrac12B_0)
\leq 4(200)^{n+1}\,\mu(\tfrac12B_0).$$
\end{proof}
\vv

\subsection{The bad balls and the function $d(\cdot)$}

We consider the constant
\begin{equation}\label{eqchoa}
A = 4\frac{\omega(5B_0)}{\mu(\tfrac12B_0)}.
\end{equation}
Notice that, by Lemma \ref{lemb0}, 
$$A\approx h(x_0)^{-1}.$$

For each $x\in \partial\Omega\cap 2B_0$ and $r\in (0,r(B_0)]$, we say that the ball $B(x,r)$ is bad (and we write $B(x,r)\in \bad$) if 
$$\omega(B(x,r))>A\,\mu(B(x,10r)).$$
Given some fixed parameter $\rho_0\in (0,\frac1{10}r(B_0)]$,
if there exists some $r\in (\rho_0,r(B_0)]$ such that $B(x,r)$ is bad, we denote
\begin{equation}\label{eqsup83}
r_0(x) = \sup\big\{ r\in (\rho_0,r(B_0)]: B(x,r)\mbox{ \;is bad}\big\}.
\end{equation}
Otherwise, we set
$$r_0(x) = \rho_0.$$
Using the openness of the balls in the definition of $r_0(x)$, it is easy to check that the supremum
in \rf{eqsup83} is attained and thus the ball $B(x,r_0(x))$ is bad if $r_0(x)>\rho_0$.

Next we define t¡he following regularized version of $r_0(\cdot)$:
$$d(x) = \inf_{y\in 2B_0\cap\partial\Omega} (r_0(y) + |x-y|),\quad\mbox{ for $x\in\R^{n+1}$}.$$
It is immediate to check that $d(\cdot)$ is a $1$-Lipschitz function. Further, since $r_0(x)\leq r(B_0)$ for any $x\in 2B_0\cap\partial\Omega$, we infer that 
$$d(x)\leq r(B_0)\quad\mbox{ for any $x\in 2B_0\cap\partial\Omega$}$$ 
too.

We need the following auxiliary result.

\begin{lemma}\label{lemupom}
Let $x\in 2B_0\cap\partial\Omega$. For all $r\in[d(x),r(B_0)]$,
\begin{equation}\label{eqom44}
\omega(B(x,r)) \leq A\,\mu(B(x,32r)).
\end{equation}
\end{lemma}

\begin{proof}
Suppose first that $r\geq r(B_0)/3$. In this case, using just that $B(x,r)\subset 3B_0$, $B_0\subset B(x,3r(B_0))$, and the choice of $A$ in \rf{eqchoa}, we infer that
$$\omega(B(x,r))\leq \omega(3B_0)\leq A\,\mu(B_0)\leq A\,\mu(B(x,3r(B_0)))\leq A\,\mu(B(x,9r)).$$

Assume now that  $r< r(B_0)/3$. Let $y\in2B_0\cap\partial\Omega$ be such that
$$2d(x) \geq r_0(y) + |x-y|.$$
Using that $B(x,r)\subset B(y,|x-y|+r)\subset B(y,3r)$ (because $|x-y|\leq 2d(x)\leq 2r$) and that
$$3r\geq 3d(x)\geq \frac32 r_0(y) \quad\mbox{ and }\quad 3r\leq r(B_0),$$ 
we get
$$\omega(B(x,r))\leq \omega(B(y,3r))\leq A\,\mu(B(y,30r)).$$
Now we take into account that $B(y,30r)\subset B(x,|x-y|+30r)\subset B(x,32r)$ (again because $|x-y|\leq 2r$), and we derive
$$\omega(B(x,r))\leq A\,\mu(B(y,30r))\leq A\,\mu(B(x,32r)).$$
\end{proof}

\vv
Now we apply Vitali's $5r$-covering theorem to get a finite subfamily of balls 
\begin{equation}\label{eqbii}
\{B_i\}_{i\in I}\subset
\{B(x,\tfrac1{2000}d(x))\}_{x\in 2B_0\cap\partial\Omega}
\end{equation}
 such that
\begin{itemize}
\item the balls $B_i$, $i\in I$, are pairwise disjoint, and
\item $\bigcup_{x\in 2B_0\cap\partial\Omega} B(x,\tfrac1{2000}d(x)) \subset \bigcup_{i\in I} 5B_i.$
\end{itemize}

In the next lemma we show some elementary properties of the family $\{B_i\}_{i\in I}$.

\begin{lemma}\label{lemballs}
Let $\{B_i\}_{i\in I}$ be the family of balls defined above. The following holds:
\begin{itemize}
\item[(a)] For each $i\in I$, $r(B_i)\leq \tfrac1{2000}r(B_0)$ and $1000B_i\subset 3B_0$.

\item[(b)] For all $x\in 1000B_i$, with $i\in I$, 
$1000\,r(B_i)\leq d(x)\leq 3000\,r(B_i).$

\item[(c)] If $1000B_i\cap 1000B_j\neq\varnothing$, for $i,j\in I$, then
$\frac13r(B_i)\leq r(B_j)\leq 3r(B_i)$.

\item[(d)] The balls $1000B_i$, $i\in I$, have finite superposition. That is,
$$\sum_{i\in I}\chi_{1000B_i}\leq C_1,$$
for some absolute constant $C_1$.

\end{itemize}
\end{lemma}

\begin{proof}
Denote by $x_i$ the center of $B_i$, so that $B_i=B(x_i,\tfrac1{2000}d(x_i))$.
The statement in (a) is due to the fact that, for each $i\in I$, we have
$$r(B_i) = \frac1{2000}d(x_i) \leq \frac1{2000}r_0(x_i)\leq \frac1{2000}r(B_0),$$
with $x_i\in 2B_0$.

On the other hand, notice that, for all $x\in 1000B_i$,
$$|d(x)-d(x_i)|\leq  |x-x_i|\leq \frac{1000}{2000}d(x_i),$$
and thus
$$\frac12d(x_i)\leq d(x)\leq \frac32\,d(x_i),$$
which gives (b).

Concerning (c), given $1000B_i$ and $1000B_j$ with non-empty intersection, we consider $x\in 1000B_i\cap 1000B_j$ and we deduce that
$$\frac12d(x_i)\leq d(x)\leq \frac32\,d(x_j).$$
Together with the converse estimate, this shows that
$$r(B_i)\leq 3r(B_j)\leq 9 r(B_i).$$

To prove (d), let $B_{i_1}$,\ldots, $B_{i_m}$ be such that
$$\bigcap_{j=1}^m 1000B_{i_j}\neq\varnothing.$$
Suppose that $B_{i_1}$ has maximal radius among the balls $B_{i_1}$,\ldots, $B_{i_m}$, so that
$$\bigcup_{j=1}^m 1000B_{i_j}\subset 3000 B_{i_1}.$$
Since the balls $B_{i_1}$,\ldots, $B_{i_m}$ are pairwise disjoint, by the properties (c), (b) and the usual volume considerations we deduce that
$$\frac m{3^{n+1}} r(B_{i_1})^{n+1}\leq
\sum_{j=1}^m r(B_{i_j})^{n+1} \leq (3000\, r(B_{i_1}))^{n+1},$$
and thus $m\lesssim1$.
\end{proof}

\vv

\subsection{The function $\psi$}

Let $\vphi$ be a radial $C^\infty$ function such that $\chi_{B(0,1.1)}\leq \vphi\leq \chi_{B(0,1.2)}$, and let
\begin{equation}\label{eqphi1}
\vphi_i(x) = \vphi\Big(\frac{x-x_i}{5r_i}\Big),
\end{equation}
where $x_i$ is the center of $B_i$ and $r_i$ its radius. Notice that $\vphi\equiv1$ on $5.5B_i$ and vanishes out of $6B_i$.
Next we need to define some auxiliary functions $\theta_j$. First, by applying the $5r$-covering theorem, we consider a covering
of $3B_0\setminus\bigcup_{i\in I}1.1B_i$ 
with balls of the form $B(z_j,10^{-5}d(z_j))$, with $z_j\in 3B_0\setminus\bigcup_{i\in I}1.1B_i$,
so that the balls $\frac15B(z_j,10^{-5}d(z_j))$ are disjoint. This implies that the 
dilated balls $1.2B(z_j,10^{-5}d(z_j))$ have finite superposition, by arguments analogous to the ones in Lemma \ref{lemballs}. For each $j\in J$, we define
$$\theta_j(x) = \vphi\Big(\frac{x-z_j}{10^{-5}d(z_j)}\Big).$$
In this way, using the property (b) in the preceding lemma, for any $j\in J$,
\begin{equation}\label{eqsupp91}
\supp\theta_j\cap \bigcup_{i\in I} 5B_i=\varnothing.
\end{equation}

We consider the functions
$$\wt \vphi_i =\frac{\vphi_i}{\sum_{j\in I} \vphi_j + \sum_{j\in j} \theta_j},\qquad i\in I.$$
Notice that the denominator above is bounded away from $0$ in $\supp\vphi_i$,  and thus $\wt \vphi_i\in C^\infty$, with
$\|\nabla\wt\vphi_i\|_\infty\lesssim r_i^{-1}$. Further, by construction
$$0\leq\sum_{i\in I}\wt\vphi_i\leq 1\quad\mbox{in\; \;$\R^{n+1}$}.$$
Also, taking into account \rf{eqsupp91}, 
$$\sum_{i\in I}\wt\vphi_i\equiv 1\quad\mbox{in\; \;$\bigcup_{i\in I}5B_i$}$$
and, since $\supp\wt\vphi_i\subset 6B_i$,
$$\sum_{i\in I}\wt\vphi_i\equiv 0\quad\mbox{in \;\;$\R^{n+1}\setminus\bigcup_{i\in I}6B_i$.}$$
We also denote
$$\psi_0= \bigg(1- \sum_{i\in I}\wt\vphi_i\bigg)\, \vphi\bigg(\frac{x-x_{0}}{r(B_0)}\bigg)
$$
(recall that $x_{0}$ is the center of $B_0$).
Finally, we let
$$\psi = \psi_0^4.$$

\begin{lemma}\label{lemad4}
The following holds:
\begin{itemize}
\item[(a)] $\supp\psi_0\subset 2B_0\setminus \bigcup_{i\in I}5B_i$ and 
$\psi_0\equiv1$ in $B_0\setminus \bigcup_{i\in I} 6B_i$.
\item[(b)] $\supp\nabla \psi_0\subset \bigcup_{i\in I} A(x_i,5r_i,6r_i) \cup A(x_0,r(B_0),2r(B_0)).$
\item[(c)] $|\nabla \psi_0(x)|\lesssim \frac{1}{r(B_i)}\,$ for all $x\in  6B_i$.
\item[(d)] $|\nabla \psi_0(x)|\lesssim \frac{1}{r(B_0)}\,$ for all $x\in2B_0\setminus \bigcup_{i\in I} 6B_i$.
\end{itemize}
The same properties hold for $\psi$.
\end{lemma}

The proof of the lemma follows easily from the construction above, and we leave it for the reader.\vv

\vv

\subsection{The sets $V$, $\wt V$, and $F$}

By Vitali's $5r$-covering theorem, there exists a subfamily $\bad_V\subset \bad$ such that
\begin{itemize}
\item the balls from $\bad_V$ are pairwise disjoint, and
\item any ball from $\bad$ is contained in some ball $5B$, with $B\in\bad_V$.
\end{itemize}
We denote
$$V=\bigcup_{B\in\bad_V} 5B,\qquad \wt V=\bigcup_{B\in\bad_V} 10B.$$
Notice that $V\subset \wt V$ and that all the bad balls are contained in $V$ (not only the ones with radius larger than $\rho_0$).

In the next lemma we show that $\wt V$ is rather small, because of our choice of $A$ above.

\begin{lemma}\label{lemV}
We have
$$\mu\big(\wt V\big)\leq \sum_{B\in\bad_V} \mu(10B)\leq \frac14\,\mu(\tfrac12 B_0).$$
\end{lemma}

\begin{proof}
By the definition of bad balls and the disjointness of the family $\bad_V$, we get
$$\mu(\wt V) \leq \sum_{B\in\bad_V} \mu(10B)\leq \frac1A\,\sum_{B\in\bad_V} \omega(B)\leq \frac1A\,\omega(5B_0),$$
where, in the last inequality, we took into account that the bad balls are centered at $2B_0$ and have radius at most $r(B_0)$. By the choice of $A$ in \rf{eqchoa}, we are done.
\end{proof}

\vv

Next we need to consider another kind of bad set. We let $F$ be the subset of the points $x\in E\cap \frac12 B_0$ for which
there exists some $r\in(0,\frac1{4}r(B_0)]$ such that
$$\omega(B(x,r))\leq \kappa_0\,h(x_0)^{-1}\,\mu(B(x,r))$$
(recall that $\kappa_0\in(0,1/10)$ will be fixed below).

\begin{lemma}\label{lemff}
We have
$$\mu(F)\leq C\kappa_0\, \mu(\tfrac12B_0).$$
\end{lemma}

\begin{proof}
By the Besicovitch covering theorem, there exists a covering of $F$ by a family of balls $B(z_i,s_i)$, with $z_i\in F$, $0<s_i\leq r(B_0)/4$, such that $\omega(B(z_i,s_i))\leq \kappa_0\,h(x_0)^{-1}\,\mu(B(z_i,r_i)),$ and having
finite superposition. That is, $\sum_i\chi_{B(z_i,s_i)}\lesssim1$.
Then we have:
\begin{align*}
\omega(F) & \leq\sum_i \omega(B(z_i,s_i)) \leq \kappa_0\,h(x_0)^{-1}\,\sum_i \mu(B(z_i,s_i)) \\
&\leq C\,
\kappa_0\,h(x_0)^{-1}\,\mu(B_0)\leq C\,
\kappa_0\,h(x_0)^{-1}\,\mu(\tfrac12B_0),
\end{align*}
taking into account that all the balls $B(z_i,s_i)$ are contained in $B_0$ and the finite superposition
of the balls in the before to last inequality, and the fact that $\frac12B_0$ is doubling with respect to 
$\mu$, by \rf{eqsau3}, in the last inequality.
As a consequence,
\begin{align*}
\mu(F) & = \int_F h\,d\omega \leq h(x_0)\,\omega(F) + \int_{\frac12B_0} |h-h(x_0)|\,d\omega\\
&\leq C\kappa_0\,\mu(\tfrac12B_0) + \kappa_0\,h(x_0)\,\omega(\tfrac12B_0)\leq C\,\kappa_0\,\mu(\tfrac12B_0).
\end{align*}
\end{proof}

\vv

\vv


\section{Proof of Theorem \ref{teo1}}

Let $s>n$ be as in Theorem \ref{teo1}.
Recall that we assume $s\in(n,n+1)$ and we denote $a=s-n$. Also, we take 
$$\alpha = \frac{1-a}{1+a},$$
so that both $a,\alpha\in (0,1)$.
We will apply the identity \rf{eqa8} with $u$ equal to the Green function $g(\cdot,p)$, the function $\psi$ constructed
in Section \ref{secpsi}, and the preceding value of $\alpha$. Recall that $\psi$ supported
in $2B_0$ and vanishes in a neighborghood of $\partial\Omega\cap 2B_0$. Thus, $g=g(\cdot,p)$ is harmonic in $\supp\psi$.
Recall also that we have
\begin{align}\label{eqa88}
 |\alpha (\alpha-1&)| \int |\nabla g|^4\,
g^{\alpha-2}\,\psi\,dx \\
& \leq 
 \left|\int \nabla(|\nabla g|^2)\cdot\nabla\psi\, g^\alpha\,dx\right| +
\left|\int |\nabla g|^2\,\nabla(g^\alpha)\cdot\nabla\psi\,dx\right| - 2\int |\nabla^2 g|^2 \,g^\alpha\,\psi\,dx.\nonumber
\end{align}
To achieve the desired contradiction to prove Theorem \ref{teo1} we will show that the integral on the left hand side
tends to $\infty$ as $\rho_0\to0$, while the right hand side is much smaller than the left hand side.

We denote
\begin{align*}
I_0 & = \int |\nabla g|^4\,g^{\alpha-2}\,\psi\,dx,\\
I_1 & = \int \nabla(|\nabla g|^2)\cdot\nabla\psi\, g^\alpha\,dx,\\
I_2 &= \int |\nabla g|^2\,\nabla(g^\alpha)\cdot\nabla\psi\,dx,\\
I_3 &= \int |\nabla^2 g|^2 \,g^\alpha\,\psi\,dx.
\end{align*}

\vv


\subsection{Estimate of $I_1$}

Using the fact that $|\nabla(|\nabla g|^2)| \lesssim |\nabla^2g|\,|\nabla g|$ and H\"older's inequality
and recalling that $\psi=\psi_0^4$, we 
get
\begin{align*}
|I_1| & \lesssim \int |\nabla^2 g|\,|\nabla g|\,g^\alpha \,\psi_0^3\,|\nabla\psi_0|\,dx
\leq \left( \int |\nabla^2 g|^2\,g^\alpha\,\psi_0^4\,dx\right)^{1/2}
\left(\int |\nabla g|^2\,g^\alpha\,\psi_0^2\,|\nabla\psi_0|^2\,dx\right)^{1/2}.
\end{align*}
Observe that the first integral on the right hand side coincides with $I_3$. To deal with the last one, we split it as
follows:
\begin{equation}\label{eqnap1}
\int |\nabla g|^2\,g^\alpha\,\psi_0^2\,|\nabla\psi_0|^2\,dx \leq \sum_{i\in I}\int_{6B_i}\ldots\; + \int_{2B_0\setminus
\bigcup_{i\in I} 6B_i}\ldots.
\end{equation}
By Lemma \ref{lemad4} (c) and Caccioppoli's inequality, for each $i\in I$ we obtain
\begin{align}\label{eqnap2}
\int_{6B_i} |\nabla g|^2\,g^\alpha\,|\nabla\psi_0|^2\,dx & \lesssim \frac1{r_i^2}\,\sup_{6B_i}g(x)^\alpha\int_{6B_i} |\nabla g|^2\,dx\\
& \lesssim \frac1{r_i^4}\,\sup_{6B_i}g(x)^\alpha\int_{12B_i} |g|^2\,dx
\lesssim r_i^{n-3}\,\sup_{12B_i}g(x)^{\alpha+2}
.\nonumber
\end{align}
By \rf{eqgreen2}, we have
\begin{equation}\label{eqnap3}
\sup_{12B_i}g(x) \lesssim \frac{\omega(96B_i)}{r_i^{n-1}}.
\end{equation}
Therefore, by the choice of $\alpha$,
\begin{equation}\label{eqnap4}
\int_{6B_i} |\nabla g|^2\,g^\alpha\,|\nabla\psi_0|^2\,dx\lesssim 
r_i^{n-3}\,\left(\frac{\omega(96B_i)}{r_i^{n-1}}\right)^{\alpha+2} = \omega(96B_i)\left(\frac{\omega(96B_i)}{r_i^s}\right)^{\alpha+1}.
\end{equation}
By Lemma \ref{lemupom},
\begin{equation}\label{eqnap5}
\omega(96B_i)\leq \omega(B(x_i,d(x_i))\leq A\,\mu(B(x_i,32d(x_i)) \lesssim Ad(x_i)^s\approx Ar_i^s.
\end{equation}
Thus,
\begin{equation}\label{eqnap6}
\int_{6B_i} |\nabla g|^2\,g^\alpha\,|\nabla\psi_0|^2\,dx\lesssim A^{\alpha+1}\omega(96B_i).
\end{equation}
Finally, Lemma \ref{lemballs} (a) and (d),
\begin{equation}\label{eqnap7}
\sum_{i\in I}\int_{6B_i} |\nabla g|^2\,g^\alpha\,|\nabla\psi_0|^2\,dx\lesssim A^{\alpha+1}
\sum_{i\in I}
\omega(96B_i)\lesssim A^{\alpha+1}\,\omega(3B_0)\lesssim A^{\alpha+2}\mu(\tfrac12B_0).
\end{equation}

\vv
Next we deal with the last integral on the right hand side of \rf{eqnap1}. We argue as in 
\rf{eqnap1}-\rf{eqnap7}, but now we use the fact that $|\nabla\psi_0|\lesssim 1/r(B_0)$ in
$2B_0\setminus \bigcup_{i\in I} 6B_i$ and we replace $6B_i$ by $2B_0$. Then, as in \rf{eqnap4}, we get
\begin{equation}\label{eqnap8}
\int_{2B_0\setminus
\bigcup_{i\in I} 6B_i}|\nabla g|^2\,g^\alpha\,|\nabla\psi_0|^2\,dx\lesssim \omega(32B_0)\left(\frac{\omega(32B_0)}{r(B_0)^s}\right)^{\alpha+1}.
\end{equation}
Using now \rf{eqsau3} and \rf{eqchoa}, we derive
\begin{equation}\label{eqnap9}
\int_{2B_0\setminus
\bigcup_{i\in I} 6B_i}|\nabla g|^2\,g^\alpha\,|\nabla\psi_0|^2\,dx\lesssim A^{\alpha+2}\mu(\tfrac12B_0).
\end{equation}
\vv

Altogether, we obtain
\begin{equation}\label{eqI2}
|I_1|\leq \bigl(CA^{\alpha+2}\mu(B_0)\bigr)^{1/2} \,I_3^{1/2} \leq CA^{\alpha+2}\mu(\tfrac12B_0) +
I_3.
\end{equation}
\vv


\subsection{Estimate of $I_2$}
Using that $\nabla (g^\alpha) = \alpha\,g^{\alpha-1}\nabla g$, $\nabla\psi=4\psi_0^3\,\nabla\psi_0$, and H\"older's inequality, we get
$$
|I_2| \lesssim \int |\nabla g|^3\,g^{\alpha-1}\,\psi_0^3\,|\nabla\psi_0|\,dx\leq
\left(\int |\nabla g|^4 g^{\alpha-2}\,\psi_0^4\,dx\right)^{3/4} 
\left(\int g^{\alpha+2}\,|\nabla\psi_0|^4\,dx\right)^{1/4}.
$$
Observe that the first integral on the 
right hand side equals $I_0$. To estimate the second one we split it:
\begin{equation}\label{eqnapp1}
\int g^{\alpha+2}\,|\nabla\psi_0|^4\,dx
\leq \sum_{i\in I}\int_{6B_i}\ldots\; + \int_{2B_0\setminus
\bigcup_{i\in I} 6B_i}\ldots.
\end{equation}
By Lemma \ref{lemad4} (c), for each $i\in I$, we obtain
\begin{align}\label{eqnapp2}
\int_{6B_i} g^{\alpha+2}\,|\nabla\psi_0|^4\,dx & \lesssim r_i^{n-3}\,\sup_{6B_i}g(x)^{\alpha+2} .\nonumber
\end{align}
As in \rf{eqnap3}, we have
\begin{equation*}
\sup_{6B_i}g(x)\leq \sup_{12B_i}g(x) \lesssim \frac{\omega(96B_i)}{r_i^{n-1}}.
\end{equation*}
Then, operating exactly as in \rf{eqnap4}-\rf{eqnap7}, we derive
$$\sum_{i\in I}\int_{6B_i}  g^{\alpha+2}\,|\nabla\psi_0|^4\,dx\lesssim A^{\alpha+2}\mu(\tfrac12B_0).$$

\vv

To estimate the last integral on the right hand side of \rf{eqnapp1} we use the fact that $|\nabla\psi_0|\lesssim 1/r(B_0)$ in
$2B_0\setminus \bigcup_{i\in I} 6B_i$ and we apply \rf{eqgreen2}. Then we get
$$\int_{2B_0\setminus
\bigcup_{i\in I} 6B_i} g^{\alpha+2}\,|\nabla\psi_0|^4\,dx\lesssim r(B_0)^{n-3}\,\sup_{2B_0}g(x)^{\alpha+2}
\lesssim \omega(32B_0)\left(\frac{\omega(32B_0)}{r(B_0)^s}\right)^{\alpha+1},
$$
which is the same estimate as in \rf{eqnap8}. Then, as in \rf{eqnap9}, we deduce
$$\int_{2B_0\setminus
\bigcup_{i\in I} 6B_i}g^{\alpha+2}\,|\nabla\psi_0|^4\,dx\lesssim A^{\alpha+2}\mu(\tfrac12B_0).$$
\vv
Therefore,
$$|I_2| \leq I_0^{3/4}\,(A^{\alpha+2}\mu(\tfrac12B_0))^{1/4} \leq \frac{|\alpha(1-\alpha)|}2\,I_0 + C(\alpha)A^{\alpha+2}\mu(\tfrac12B_0).
$$
\vv


\subsection{Lower estimate of $I_0$}

From the identity \rf{eqa88} and the estimates for $I_1$ and $I_2$ we derive
\begin{align*}
|\alpha (\alpha-1)|\,I_0 &= |\alpha (\alpha-1)|\int |\nabla g|^4\,
g^{\alpha-2}\,\psi\,dx \\
&\leq I_1+ I_2 - 2I_3 \\
& \leq \bigg(CA^{\alpha+2}\mu(\tfrac12B_0) +
 I_3\bigg) + \bigg(\frac{|\alpha(1-\alpha)|}2\,I_0 + C(\alpha)A^{\alpha+2}\mu(\tfrac12B_0)\bigg) - 2I_3
.
\end{align*}
Hence,
\begin{equation}\label{eqi01}
I_0\leq C(\alpha)A^{\alpha+2} \mu(\tfrac12B_0).
\end{equation}
In this section, by estimating $I_0$ from below, we will contradict this inequality.
\vv

To get a lower estimate for $I_0$ we need to define some reasonably good set contained in $\frac12B_0\cap\partial\Omega$.
To this end, we need first to introduce another type of balls. 
Let $I_b\subset I$ be the subfamily of indices $i$ such that $r(B_i)>\frac1{2000}\,\rho_0$. Recall that $I$ is the set of indices in \rf{eqbii} and $r(B_i)=\frac1{2000}d(x_i)\leq 
\frac1{2000}r_0(x_i)$. So if $i\in I_b$, then $r_0(x_i)>\rho_0$ and thus $B(x_i,r_0(x_i))$ is a bad ball.
We say that a ball $B$ is useless (and we write $B\in \uss$) if it is centered at $\frac12 B_0\cap\partial
\Omega\setminus \wt V$ and
\begin{equation}\label{eqassu88}
\mu\bigg(\bigcup_{i\in I_b:6B_i\cap B\neq\varnothing} 960B_i\bigg) > \ve_1\,\mu(B)
\qquad\mbox{and}\qquad \mu(B)\geq 3^{-s}r(\tfrac12B)^s,
\end{equation}
where $\ve_1\in(0,1/10)$ is a small parameter to be fixed below that will depend only on $n$.

Recall now that, by Lemma \ref{lemV},
$$\sum_{B\in\bad_V} \mu(10B)\leq \frac14\,\mu(\tfrac12 B_0).$$
Hence there exists some $\rho_1>0$ such that
\begin{equation}\label{eqvv}
\sum_{B\in\bad_V:r(B)\leq \rho_1} \mu(10B)\leq \ve_1^2\,\mu(\tfrac12 B_0).
\end{equation}
Notice that $\rho_1$ may depend here on the particular measure $\mu$, not only on $n$.
We define
$$U(\rho_1) = \bigcup_{B\in\uss:\,r(B)\leq \rho_1} B.$$

\begin{lemma}\label{lemuss}
We have
$$\mu(U(\rho_1))\lesssim \ve_1\,\mu(\tfrac12 B_0).$$
\end{lemma}

\begin{proof}
Let $B\in\uss$ with $r(B)\leq\rho_1$ and let $B_i$, $i\in I_b$, be such that $6B_i\cap B\neq \varnothing$. 
Notice that $2000B_i$ is contained in some bad ball (because $d(x_i)>\rho_0$), which in turn is
contained in some ball $5B'$, with $B'\in\bad_V$. Thus, $2000B_i\subset 5B'$. Now note that $B$ is centered at $\partial\Omega\setminus \wt V\subset (10B')^c$, and observe that the condition $6B_i\cap B\neq \varnothing$
implies that $B$ intersects $5B'$. These two facts ensure that 
\begin{equation}\label{eqcad48}
r(B)\geq r(5B') \geq r(2000B_i).
\end{equation}
Then we deduce that
$$960B_i\subset2000B_i\subset 3B.$$
The first inequality in \rf{eqcad48} also implies that $r(B')\leq\rho_1$, which in turn gives that
$$960B_i\subset \bigcup_{B''\in \bad_V:r(B'')\leq \rho_1} 5B''\subset \bigcup_{B''\in \bad_V:r(B'')\leq \rho_1} 10B''=: \wt V_0,$$
with 
\begin{equation}\label{eqvv1}
\mu(\wt V_0)\leq \ve_1^2\,\mu(\tfrac12 B_0),
\end{equation}
by \rf{eqvv}.
From the first condition in \rf{eqassu88}, we deduce that
$$
\mu(3B\cap \wt V_0) \geq 
\mu\bigg(\bigcup_{i\in I_b:6B_i\cap B\neq\varnothing} 960B_i\bigg)>\ve_1\,\mu(B).
$$
Using also the fact that $\mu(15B)\lesssim r(B)^s\lesssim \mu(B)$ (by the second condition in
\rf{eqassu88}), we get
\begin{equation}\label{eqvv2}
\mu(3B\cap \wt V_0) \geq c\,\ve_1\,\mu(15B).
\end{equation}

Now we apply the $5r$-covering theorem to get a subfamily $I_U$ from the balls in $\uss$ with radius not exceeding $\rho_1$ such that
\begin{itemize}
\item the balls $3B$ with $B\in I_U$ are pairwise disjoint, and
\item $U(\rho_1)\subset \bigcup_{B\in I_U} 15B$.
\end{itemize}
From these properties and \rf{eqvv2} and \rf{eqvv1}, we obtain
$$\mu(U(\rho_1))\leq \sum_{B\in I_U} \mu(15B) \lesssim \frac1{\ve_1}\sum_{B\in I_U} \mu(3B\cap \wt V_0)
\leq \frac1{\ve_1}\,\mu(\wt V_0) \leq \ve_1\,\mu(\tfrac12 B_0).$$
\end{proof}

\vv

Now we are ready to define the aforementioned reasonably good set contained in $\frac12B_0\cap\partial\Omega$. First denote 
$$G_0= \tfrac12B_0\cap\partial\Omega \setminus \big( F \cup \wt V\cup U(\rho_1)\big),$$
and recall that, by Lemmas  \ref{lemff}, \ref{lemV}, and \ref{lemuss},
$$\mu(G_0)\geq \mu(\tfrac12B_0) - C\kappa_0\,\mu(\tfrac12B_0) - \frac14\,\mu(\tfrac12B_0)-
C\ve_1\,\mu(\tfrac12B_0).$$
We assume $\kappa_0$ to be an absolute  constant small enough so that $C\kappa_0\leq1/4$
and also $\ve_1$ small enough so that $C\ve_1\leq1/4$, and then we obtain
$$\mu(G_0)\geq \frac14\,\mu(\tfrac12B_0).$$

Next we need to define some families of balls centered at $G_0$ inductively.
Let $G$ be the subset of those $x\in G_0$ such that $\Theta^{s,*}(x,\mu)\geq 2^{-s}$. Note that, by \rf{eqdens*}, 
$\mu(G_0\setminus G)=0$.
By definition, for each $\eta_k\in (0,r(B_0)/10]$, for $\mu$-a.e.\ $x\in G$ there exists a ball $B_x^i$ centered at $x$ with radius $r(B_x^i)\leq \eta_k$ such that
\begin{equation}\label{eqmu99}
\mu(B_x^i) \geq 3^{-s} r(B_x^i)^s.
\end{equation}
Hence, by the $5r$-covering theorem, we can extract a subfamily  $\wt\FF_k\subset\{2B_x^i\}_{x\in G}$ such that
\begin{itemize}
\item[(a)] $G\subset \bigcup_{B\in \wt \FF_k} 80B,$ and
\item[(b)] the balls $16B$, $B\in \wt \FF_k$,  are disjoint.
\end{itemize}
Further, we can still extract a {\em finite} subfamily $\FF_k\subset \wt\FF_k$ such that
\begin{equation}\label{eqmu995}
\mu\bigg(\bigcup_{B\in  \FF_k} 80B\bigg)\geq \frac12 \,\mu(G) \geq \frac18\,\mu(\tfrac12 B_0).
\end{equation}
Now we fix inductively the parameters $\eta_k$ as follows: first we take $\eta_1 = r(B_0)/10$,
and next we set
$$\eta_{k} = \ve_0\,\min_{B\in\FF_{k-1}}r(B),$$
where $\ve_0\in(0,1/100)$ is some small constant to be chosen below. Notice that this choice ensures that the balls from
the family $\FF_k$ are much smaller than the ones of the preceding families $\FF_1,\ldots,\FF_{k-1}$.

\vv
Remark that the balls $B(x,r)$ centered at $G$ (like the balls from the families $\FF_k$) satisfy a couple of crucial properties:
\begin{itemize}
\item For $r\in(0,r(B_0)/4]$,
\begin{equation}\label{eqom99}
\omega(B(x,r))\geq \kappa_0\,h(x_0)^{-1}\,\mu(B(x,r)),
\end{equation}
by the construction of $F$ and $G$.
\item For $r\in [\rho_0,r(B_0)]$,
\begin{equation}\label{eqom99'}
\omega(B(x,r)) \leq A\,\mu(B(x,10r))\leq CA\,r^s\approx h(x_0)^{-1} r^s,
\end{equation}
taking into account that $B(x,r)$ does not belong to $\bad$ and using \rf{eqom44} and the choice of $A$. 
\end{itemize}

\vv
The next lemma contains the key estimate that will allow to bound $I_0$ from below.

\begin{lemma}\label{lemesti0}
Let $B\in\FF_k\setminus\uss$ and suppose that $\rho_0\leq \frac12 \eta_{k+1}$. 
Denote
\begin{equation}\label{eqwtb1}
\wt B= B\setminus \bigg(\bigcup_{B'\in \FF_{k+1}}B' \cup \bigcup_{i\in I} 6B_i\bigg).
\end{equation}
Then,
\begin{equation}\label{eqfifi9}
\int_{\wt B} |\nabla g|^4\,g^{\alpha-2}\,\psi\,dx \gtrsim A^{\alpha+2}\mu(B).
\end{equation}
\end{lemma}

\begin{proof}
Denote by $x_B$ the center of $B$ and let
$$\vphi_B(y) = \vphi\Big(\frac{y-x_B}{\tfrac12r(B)}\Big),$$
where $\vphi$ is the radial $C^\infty$ function appearing in \rf{eqphi1}, so that $\supp\vphi_B\subset B$, $\vphi\equiv1$ in $\frac12B$, and
$\|\nabla\vphi_B\|_\infty\lesssim1/r(B)$.
Then we have
$$\frac1{r(B)} \int_{B} |\nabla g|\,dx \gtrsim  \left|\int \nabla g\cdot \nabla\vphi_B\,dx\right| = \int \vphi_B\,d\omega
\geq \omega(\tfrac12B) \geq  \kappa_0\,h(x_0)^{-1}\,\mu(\tfrac12B),$$
taking into account \rf{eqom99} for the last inequality.

Next we will show that $\frac1{r(B)}\int_{B\setminus \wt B} |\nabla g|\,dx$ is small if the parameters
$\ve_0$ and $\ve_1$ above are small. To this end, given any ball $B'$ intersecting $B$ and centered at $x_{B'}\in B_0\cap \partial\Omega$ with radius $r(B')\in [d(x_{B'}), r(B)]$, we write:
\begin{align*}
\int_{B\cap B'} |\nabla g|\,dx \lesssim {r(B')^n}\,
\sup_{2B'}g(x),
\end{align*}
applying H\"older's inequality and Caccioppoli's inequality.
By \rf{eqgreen2}, the last supremum can be bounded above by $\omega(16B')\,r(B')^{1-n}$, and then we get
\begin{equation}\label{eqfac43}
\int_{B\cap B'} |\nabla g|\,dx \lesssim
r(B')\,\omega(16B').
\end{equation}

We split
\begin{equation}\label{eqspl93}
\int_{B\setminus \wt B} |\nabla g|\,dx \leq
\int_{B\cap \bigcup_{B'\in \FF_{k+1}}B'}|\nabla g|\,dx  + \int_{B\cap \bigcup_{i\in I\setminus I_b} 6B_i} |\nabla g|\,dx 
+\int_{B\cap \bigcup_{i\in I_b} 6B_i}|\nabla g|\,dx.
\end{equation}
To deal with the first integral on the right hand side, recall that the balls $16B'$, with $B'\in \FF_{k+1}$,  are disjoint and  their radius is at most 
$\ve_0\,r(B)$, by construction. Thus,
$$\frac1{r(B)}\int_{B\cap \bigcup_{B'\in \FF_{k+1}}B'}|\nabla g|\,dx \lesssim \frac1{r(B)}
\sum_{B'\in \FF_{k+1}:16B'\subset 2B}
r(B')\,\omega(16B') \lesssim \ve_0\,\omega(2B).$$
The second integral on the right hand side of \rf{eqspl93} is estimated analogously. In this case we use that all the balls $B_i$ with $i\in I\setminus I_b$ have radius equal to $\rho_0/2000\ll r(B)$, and that the balls $96B_i$, $i\in I\setminus I_b$,
have bounded overlap, by Lemma \ref{lemballs} (d). Then, by \rf{eqfac43}, we deduce
$$\frac1{r(B)}\int_{B\cap \bigcup_{i\in I\setminus I_b} 6B_i}|\nabla g|\,dx \lesssim \frac1{r(B)}
\sum_{i\in I\setminus I_b: 96B_i\subset 2B}
r(B_i)\,\omega(96B_i) \lesssim \ve_0\,\omega(2B).$$

Regarding the last integral on the right hand side of \rf{eqspl93}, we will use the fact that
\begin{equation}\label{eqassu9}
\mu\bigg(\bigcup_{i\in I_b:6B_i\cap B\neq\varnothing} 960B_i\bigg) \leq \ve_1\,\mu(B),
\end{equation}
because $B$ is assumed to be not useless. Note that given $i\in I_b$ such that $B\cap 6B_i\neq \varnothing$, we have $r(6B_i)\leq r(B)$. Otherwise, $B\subset18B_i$, which violates the condition \rf{eqassu9}. Then,
applying \rf{eqfac43} again, we deduce
\begin{align*}
\frac1{r(B)}\int_{B\cap \bigcup_{i\in I_b} 6B_i}|\nabla g|\,dx & \lesssim \frac1{r(B)}
\sum_{i\in I_b: 6B_i\cap B\neq\varnothing}
r(B_i)\,\omega(96B_i) \lesssim 
\sum_{i\in I_b: 6B_i\cap B\neq\varnothing}
\omega(96B_i) \\
&\lesssim  A\!\sum_{i\in I_b: 6B_i\cap B\neq\varnothing}\!\!\!
\mu(960B_i) \lesssim A \,\mu\bigg(\bigcup_{i\in I_b: 6B_i\cap B\neq\varnothing} \!\!\!960B_i\bigg)\lesssim A\ve_1\mu(2B),
\end{align*}
by the finite superposition of the balls $960B_i$ and \rf{eqassu9}.

We take into account now that, by \rf{eqom99'},  the $s$-growth of $\mu$, and \rf{eqmu99},
$$\omega(2B)\lesssim h(x_0)^{-1}\,r(2B)^s \approx h(x_0)^{-1}\,r(\tfrac12B)^s \lesssim 
h(x_0)^{-1}\,\mu(\tfrac12B).$$
Therefore,
$$\frac1{r(B)}\int_{B\cap \bigcup_{B'\in \FF_{k+1}}B'} |\nabla g|\,dx 
+ \frac1{r(B)}\int_{B\cap \bigcup_{i\in I\setminus I_b} 6B_i}|\nabla g|\,dx 
\lesssim \ve_0\,h(x_0)^{-1}\,\mu(\tfrac12B).$$
Then, using also that $\mu(\tfrac 12B)\gtrsim r^s \gtrsim \mu(2B)$, $h(x_0)^{-1}\approx A$, and recalling the splitting \rf{eqspl93}, we deduce
\begin{align*}
\frac1{r(B)}\int_{\wt B} |\nabla g|\,dx & \geq  (\kappa_0- C\ve_0 - C\ve_1)\,h(x_0)^{-1}\,\mu(\tfrac12B)\\
& \geq \frac{\kappa_0}2\,h(x_0)^{-1}\,\mu(\tfrac12B)\approx \kappa_0 \,A\,\mu(\tfrac12B)\approx \kappa_0 \,A\,\mu(B),
\end{align*}
assuming $\ve_0$ and $\ve_1$ small enough.

\vv

Next, applying H\"older's inequality, we get
$$\int_{\wt B} |\nabla g|\,dx
\leq \left(\int_{\wt B} |\nabla g|^4\,g^{\alpha-2}\,dx\right)^{1/4}
\left(\int_B g^{(2-\alpha)/3}\,dx\right)^{3/4}.$$
To estimate the last integral on the right hand side we take into account that, by \rf{eqgreen2} and  \rf{eqom99'},
$$g(x)\lesssim\frac{\omega(8B)}{r(B)^{n-1}}\lesssim A\,\frac{\mu(80B)}{r(B)^{n-1}} \approx A\,\frac{\mu(B)}{r(B)^{n-1}} $$
for all $x\in B$. Hence,
$$\int_B g^{(2-\alpha)/3}\,dx\lesssim A^{(2-\alpha)/3}\mu(B)^{(2-\alpha)/3}r(B)^{n+1-(2-\alpha)(n-1)/3}.$$
Therefore,
\begin{align*}
\int_{\wt B} |\nabla g|^4\,g^{\alpha-2}\,dx &\gtrsim 
\bigg(\int_{\wt B} |\nabla g|\,dx\bigg)^4\left(A^{(2-\alpha)/3}\mu(B)^{(2-\alpha)/3}r(B)^{n+1-(2-\alpha)(n-1)/3}\right)^{-3}\\
& \gtrsim \big(\kappa_0 \,A\,r(B)\,\mu(B)\big)^4\left(A^{(2-\alpha)/3}\mu(B)^{(2-\alpha)/3}r(B)^{n+1-(2-\alpha)(n-1)/3}\right)^{-3}\\
& = \kappa_0^4  A\mu(B)\,\bigg(\frac{A\mu(B)}{r(B)^s}\bigg)^{1+\alpha}\approx \kappa_0^4  A^{2+\alpha}\mu(B).
\end{align*}
To finish the proof of the lemma it just remains to notice that $\kappa_0$ is some absolute constant
depending just on $n$ and that $\psi =1$ on $\wt B$.
\end{proof}

\vv

Now we are ready to obtain the lower estimate for $I_0$ required to complete the proof of Theorem
\ref{teo1}. Let $N>1$ be an arbitrarily large integer, and choose $\rho_0\in (0,\frac12\eta_{N+1}]$ such that $\rho_0\ll\rho_1$ too.
Let $k_0$ be the minimal integer such that $2\eta_{k_0}\leq \rho_1$.
If $B\in \FF_k$ with $k\in [k_0,N]$, then $r(B)\leq \rho_1$ by construction and so
$B\not\in\uss$ (since $B$ is centered in $U(\rho_1)^c$). Then we can apply Lemma \ref{lemesti0}
to deduce that
$$
\int_{\wt B} |\nabla g|^4\,g^{\alpha-2}\,\psi\,dx \gtrsim A^{\alpha+2}\mu(B),
$$
with $\wt B$ defined in \rf{eqwtb1}. Then we get
\begin{align*}
\int |\nabla g|^4\,g^{\alpha-2}\,\psi\,dx & \geq \sum_{k=k_0}^N
\sum_{B\in\FF_k} \int_{\wt B} |\nabla g|^4\,g^{\alpha-2}\,\psi\,dx  \gtrsim \sum_{k=k_0}^N
\sum_{B\in\FF_k} A^{\alpha+2}\mu(B),
\end{align*}
using the fact that the regions $\wt B$ above do not overlap.
Now, by the doubling property of the balls from $\FF_k$ and \rf{eqmu995}, for each $k$ we have
$$\sum_{B\in\FF_k} \mu(B)\approx \sum_{B\in\FF_k} \mu(80B)\geq \frac12 \,\mu(G) \geq \frac18\,\mu(\tfrac12 B_0).
$$
Thus,
$$I_0=\int |\nabla g|^4\,g^{\alpha-2}\,\psi\,dx\gtrsim (N-k_0)A^{\alpha+2}\mu(\tfrac12 B_0).$$
Taking $N$ big enough, we contradict \rf{eqi01}, as wished.


\vvv

\enlargethispage{1cm}

\end{document}